\documentclass[12pt]{amsart}
\usepackage{amssymb}
\usepackage{amsfonts}
\usepackage{amssymb,latexsym}
\usepackage{enumerate}
\usepackage{mathrsfs}
\usepackage{hyperref}
 \usepackage{color}\makeatletter
\@namedef{subjclassname@2010}{%
  \textup{2010} Mathematics Subject Classification}
\makeatother

  [2002/01/19 v2.2g %
    AMS font definitions%
  ]
\DeclareFontFamily{U}{euf}{}
\DeclareFontShape{U}{euf}{m}{n}{%
  <5><6><7><8><9>gen*eufm%
  <10><10.95><12><14.4><17.28><20.74><24.88>eufm10%
  }{}
\DeclareFontShape{U}{euf}{b}{n}{%
  <5><6><7><8><9>gen*eufb%
  <10><10.95><12><14.4><17.28><20.74><24.88>eufb10%
  }{}

  [2002/01/19 v2.2g %
    AMS font definitions%
  ]
\DeclareFontFamily{U}{msb}{}
\DeclareFontShape{U}{msb}{m}{n}{%
  <5><6><7><8><9>gen*msbm%
  <10><10.95><12><14.4><17.28><20.74><24.88>msbm10%
  }{}

  [2002/01/19 v2.2g %
    AMS font definitions%
  ]
\DeclareFontFamily{U}{msa}{}
\DeclareFontShape{U}{msa}{m}{n}{%
  <5><6><7><8><9>gen*msam%
  <10><10.95><12><14.4><17.28><20.74><24.88>msam10%
  }{}

\newtheorem{theorem}{Theorem}[section]
\newtheorem{lemma}[theorem]{Lemma}
\newtheorem{proposition}[theorem]{Proposition}
\newtheorem{corollary}[theorem]{Corollary}

\theoremstyle{definition}

\newtheorem{remark}[theorem]{Remark}

\numberwithin{equation}{section} 

\def\t{\tilde}
\def\ga{\gamma}
\def\B{\overline B}
\def\E{\overline E}

\begin{document}

\title[]
{On the Stieltjes constants and gamma functions with respect to alternating Hurwitz zeta functions}

\author{Su Hu}
\address{Department of Mathematics, South China University of Technology, Guangzhou 510640, China}
\email{mahusu@scut.edu.cn}

\author{Min-Soo Kim}
\address{Department of Mathematics Education, Kyungnam University, Changwon, Gyeongnam 51767, Republic of Korea}
\email{mskim@kyungnam.ac.kr}

\begin{abstract}
Dating back to Euler, in classical analysis and number theory, the Hurwitz zeta function 
$$
\zeta(z,q)=\sum_{n=0}^{\infty}\frac{1}{(n+q)^{z}},
$$
the Riemann zeta function $\zeta(z)$, the generalized  Stieltjes constants $\gamma_k(q)$, the Euler constant $\gamma$, Euler's gamma function $\Gamma(q)$ and the digamma function $\psi(q)$ have 
many close connections on their definitions and properties.  There are also 
many integrals, series or infinite product representations of them along the history.

In this note, we try to provide a parallel story  for  the alternating Hurwitz zeta function (also known as the  Hurwitz-type Euler zeta function)
$$\zeta_{E}(z,q)=\sum_{n=0}^\infty\frac{(-1)^{n}}{(n+q)^{z}},$$
the alternating zeta function $\zeta_{E}(z)$ (also known as the  Dirichlet's eta function $\eta(z)$), the  modified  Stieltjes constants $\tilde\gamma_k(q)$, the modified Euler constant $\tilde\gamma_{0}$, the modified  gamma function $\tilde\Gamma(q)$ and the modified digamma function $\tilde\psi(q)$ (also known as the Nielsen's $\beta$ function).
Many new integrals, series or infinite product representations of these constants and special functions have been found. 
By the way, we also get two new series expansions of $\pi:$
\begin{equation*} 
\frac{\pi^2}{12}=\frac34-\sum_{k=1}^\infty(\zeta_E(2k+2)-1)
\end{equation*}
and 
\begin{equation*} \frac{\pi}{2}= \log2+2\sum_{k=1}^\infty\frac{(-1)^k}{k!}\tilde\gamma_k(1)\sum_{j=0}^kS(k,j)j!.
\end{equation*}
\end{abstract}

\subjclass[2010]{33B15, 33E20, 11M35, 11B68, 11S80}
\keywords{Alternating Hurwitz zeta function, Stieltjes constant, Gamma function, Digamma function, Infinite series.}

\maketitle 
\section{History of the subject}
For Re$(z)>1,$  the Riemann zeta function is defined by
\begin{equation}~\label{Ri-zeta}
\zeta(z)=\sum_{n=1}^{\infty}\frac{1}{n^{z}}.
\end{equation}
This function can be analytically continued to a meromorphic function  in the
complex plane with a simple pole at $z=1$. 
The special number $\zeta(3)=1.20205\cdots$ is called Ap\'ery constant.
It is named after  Ap\'ery, who proved in 1979 that $\zeta(3)$ is irrational (see \cite{Ape}).

For Re$(z)>1$  and  $q\neq0,-1,-2,\ldots,$ in 1882, Hurwitz \cite{Hurwitz} defined the partial zeta function
\begin{equation}~\label{Hurwitz}
\zeta(z,q)=\sum_{n=0}^{\infty}\frac{1}{(n+q)^{z}}
\end{equation}
which generalized (\ref{Ri-zeta}). As (\ref{Ri-zeta}), this function can also be analytically continued to a meromorphic function  in the
complex plane with a simple pole at $z=1$.

The Hurwitz zeta function $\zeta(z,q)$ and its derivatives have a close connection with
the other important constants and special functions in number theory and analysis, such as the generalized Stieltjes constant and the Euler constant,
and the gamma functions. In fact, the generalized Stieltjes constant $\gamma_{k}(q)$ comes from the following  Laurent series expansion of $\zeta(z,q)$ around $z=1$
\begin{equation}
\zeta(z,q)=\frac{1}{z-1}+\sum_{k=0}^{\infty}\frac{(-1)^{k}\gamma_{k}(q)(z-1)^{k}}{k!}
\end{equation}
and $\gamma_{k}=\gamma_{k}(1)$ is the original Stieltjes constant in 1885 (see Stieltjes' original article \cite{St} and  Ferguson \cite{Fe}).
Letting $k=0$, we recover the Euler constant
\begin{equation}
\begin{aligned}\gamma&:=\gamma_{0}(1)=\lim_{s\to1}\left(\zeta(s)-\frac{1}{s-1}\right)\\
&=\lim_{\alpha\to\infty}\left(\sum_{n=1}^{\alpha}\frac{1}{n}-\log \alpha\right)=0.5772156649\cdots.
\end{aligned}
\end{equation}
In 1972, Berndt \cite{Berndt} proved that
\begin{equation}
\gamma_{\ell}(q)=\lim_{\alpha\to\infty}\left\{\sum_{n=0}^{\alpha}\frac{\log^{\ell}(n+q)}{n+q}-\frac{\log^{\ell+1}(\alpha+q)}{\ell+1}\right\}.
\end{equation}
In 1994, Williams and Zhang \cite{ZW} generalized the above result and established
\begin{equation}
\gamma_{\ell}(q)=\sum_{n=0}^{\alpha}\frac{\log^{\ell}(n+q)}{n+q}-\frac{\log^{\ell+1}(\alpha+q)}{\ell+1}-\frac{\log^{\ell}(\alpha+q)}{2(\alpha+q)}+\int_{\alpha}^{\infty}\B_{1}(x)f'_{\ell}(x)dx,
\end{equation}
where $\B_{1}(x)=x-[x]-\frac{1}{2}$ 
is a Bernoulli periodic function, related to the periodic Euler function (defined in Proposition \ref{pro0}) and Euler polynomials (Theorem \ref{thm3}).
Recently some other integral and series representations of the generalized Stieltjes constant $\gamma_{\ell}(q)$ have also been considered  by 
Coffey (see \cite{Co06,Co08,Co13,Co}).
In 2003, Kreminski \cite{Kreminski} generalized the Stieltjes constant  from $\gamma_{\ell}(q)$ to $\gamma_{v}(q)$ for any positive real number $v$,
named the fractional Stieltjes constant, and recently Farr, Pauli and Saidak  \cite{FPS} proved Berndt and Williams \& Zhang's formulas for $\gamma_{v}(q)$ under the Gr\"unwald-Letnikov fractional derivatives.

The gamma function $\Gamma(q)$ is originally defined by Euler from its integral representation 
\begin{equation}\label{Gamma} \Gamma(q)=\int_{0}^{\infty}t^{q-1}e^{-t}dt\end{equation}
for Re$(q)>0$. For $q\in\mathbb{R}$ and $q>0,$ the log gamma function can be represented by the derivatives 
of the Hurwitz zeta functions $\zeta(z,q)$, which may sometimes be considered as 
an alternative definition of the gamma functions (e.g. \cite[Definition 9.6.13(1)]{Cohen}),
\begin{equation}\log\Gamma(q)=\zeta'(0,q)-\zeta'(0,1)=\zeta'(0,q)-\zeta'(0).\end{equation}
The following Weierstrass--Hadamard product of $\Gamma(q)$ is well-known,
\begin{equation}\label{Hadamard}
\Gamma(q)=\frac{1}{q}e^{-\gamma q} \prod_{k=1}^{\infty}\left(e^{\frac{q}{k}}\left(1+\frac{q}{k}\right)^{-1}\right),
\end{equation}
where $\gamma$ is the Euler constant.

For $q>0$, let \begin{equation}\label{Classical2} \psi(q):=\frac{d}{dq}\log\Gamma(q)\end{equation}
be the digamma function (see \cite[Definition 9.6.13(2)]{Cohen} and \cite[p. 32]{FS}).
Let  \begin{equation}\psi^{(n)}(q):=\left(\frac{d}{dq}\right)^n\psi(q),\quad n=0,1,2,\ldots.\end{equation}
We have 
\begin{equation}\label{Classical} \psi^{(n)}(q)=(-1)^{n+1}n!\zeta(n+1,q), \quad n = 1, 2, 3,\ldots\end{equation}
(see \cite[Proposition 9.6.41]{Cohen}).

In this note, we will tell some parallel stories for  the alternating Hurwitz (or Hurwitz-type Euler) zeta functions.
The alternating Hurwitz zeta function is defined by 
\begin{equation}\label{E-zeta-def}
\zeta_E(z,q)=\sum_{n=0}^\infty\frac{(-1)^n}{(n+q)^{z}},
\end{equation}
where Re$(z)>0$ and  $q\neq0,-1,-2,\ldots.$
It can be analytically
continued to the complex plane without any pole.
Sometimes we may use the notation $J(z,q)$ instead of $\zeta_E(z,q)$ 
(e.g. Williams and Zhang \cite[p. 36, (1.1)]{WZ}). 
The function $\zeta_E(z,q)$ satisfies the difference equation and the derivative formula:
\begin{equation}\label{J-1}
\zeta_E(z,q+1)+\zeta_E(z,q)=q^{-z},
\end{equation}
\begin{equation}\label{J-2}
\frac{\partial}{\partial q}\zeta_E(z,q)=-z\zeta_E(z+1,q).
\end{equation}
Recently, the Fourier expansion and several integral representations, special values and power series expansions,
convexity properties  of $\zeta_{E}(z,q)$ have been investigated (see \cite{Cvijovic, HKK, HK2019}), and  it has been found that  $\zeta_{E}(z,q)$ can be used to represent a partial zeta function of cyclotomic fields in one version of Stark's conjectures in algebraic number theory (see \cite[p. 4249, (6.13)]{HK-G}).

In particular when $q=1,$ the function $\zeta_E(z,q)$ reduces to the alternating zeta function (also known as the \textit{Dirichlet's eta} function),
\begin{equation}\label{A-zeta}
\begin{aligned}
\zeta_E(z,1)=\zeta_E(z)&=\sum_{n=1}^\infty\frac{(-1)^{n+1}}{n^{z}}=\eta(z), \\
\zeta_E(z)&=\left(1-\frac1{2^{z-1}}\right)\zeta(z),
\end{aligned}
\end{equation}
with
\begin{equation}\label{e-zeta(1)}
\zeta_E(1)=\log2.
\end{equation}
The alternating zeta function $\zeta_E(z)$  is  a particular case of Witten's zeta functions in mathematical physics \cite[p. 248, (3.14)]{Min},
and  it has been studied and evaluated at certain positive integers by
Sitaramachandra Rao \cite{SiR} in terms of Riemann zeta values. See also \cite[p. 31, \S 7]{FS} and \cite[p. 2, (2)]{Mi}.

Since the  function $\zeta_E(z,q)$ can be  analytically continued to the whole complex plane and $\zeta_E(z,q)$ is nonsingular at $z=1,$  we  designate a modified Stieltjes constants $\t\ga_k(q)$ from
the Taylor expansion of $\zeta_E(z,q)$ at $z=1$,
\begin{equation}\label{l-s-con}
\zeta_E(z,q)=\sum_{k=0}^\infty\frac{(-1)^k\t\ga_k(q)}{k!}(z-1)^k
\end{equation}
(also see \cite[(1.18)]{Co} and \cite{St}). 
 It is easy to see that 
\begin{equation}\label{gamma0} 
\t\ga_0(q)=\zeta_E(1,q).
\end{equation}
 For simplification of the notation, let \begin{equation}\label{gammak}\t\gamma_k:=\t\gamma_k(1)\end{equation}
 for $k=0,1,2,\ldots.$

We shall prove that
$$\t\ga_0=\frac12+\frac12\sum_{j=1}^\infty(-1)^{j+1}\frac{1}{j(j+1)}$$
(see Corollary \ref{cor}) and
\begin{equation*}\begin{aligned}
\t\ga_1&=\frac12-\log2+2\sum_{k=0}^\infty\left(1+\frac12(\ga-2)\ga-\frac{\pi^2}{24} \right. \\
&\quad-\frac{(2k+1)^2\pi^2}{24}\,_3F_4\left(1,1,1;2,2,2,\frac52;-\frac{(2k+1)^2\pi^2}{4}\right) \\
&\quad\left.+\log(2k+1)\pi \big(\ga-1+\frac12\log(2k+1)\pi\big) \right)
\end{aligned}
\end{equation*}
(see Corollary \ref{cor-hyper}).
Recall that  the hypergeometric function is defined by
\begin{equation}\label{hyper-1}
_pF_{p+1}(1,1,\ldots,1;2,2,\ldots,2,c;z)=\sum_{n=0}^\infty\frac{1}{(n+1)^p}\frac{\Gamma(c)}{\Gamma(c+n)}\frac{z^n}{n!}
\end{equation}
(see \cite{Co13} and \cite{SC}).

We show that Berndt and Williams \& Zhang's formulas are also established  for  $\t\ga_\ell(q),$ that is,
$$\t\ga_\ell(q)=\lim_{\alpha\to\infty}\sum_{n=0}^\alpha(-1)^n\frac{\log^\ell(n+q)}{n+q}$$
(Corollary \ref{cor0})
and 
$$\t\gamma_\ell(q)=\sum_{n=0}^\alpha(-1)^n\frac{\log^\ell(n+q)}{n+q}-(-1)^\alpha\frac{\log^\ell(\alpha+q)}{2(\alpha+q)}
+\frac12\int_\alpha^\infty\E_0(-t)g'_\ell(t)dt,$$
where $\E_0(t)$ is the 0-th periodic Euler function (see Theorem \ref{thm0}).

For $q>0$, define the modified digamma function $\t\psi(q)$ to be
\begin{equation}\label{psi-def}
\t\psi(q):=-\t\ga_0(q)=-\zeta_E(1,q).
\end{equation}
or equivalently define
\begin{equation}\label{ps-ga}
\begin{aligned}
\t\psi(q)&=-\frac{\Gamma'(q)}{\Gamma(q)}+\frac{\Gamma'(q/2)}{\Gamma(q/2)}+\log2
\\&=-\psi(q)+\psi(q/2)+\log2,
\end{aligned}
\end{equation}
where $\Gamma$ is the gamma function and  $\psi$ is the digamma function (see \cite[Proposition 2]{WZ} or the explanation in the next subsection on digamma functions).
Since $(\ref{psi-int})$ below implies  (\ref{ps-ga}), both definitions for $\t\psi(q)$ are equivalent (see (\ref{eq-defi})) (for details, we refer to the second paragraph of p. 7). 
We also  mention that $\t\psi(q)$  is sometimes named Nielsen's $\beta$ function  (e.g. \cite[p. 264, (2.1)--(2.4)]{Na}),
and a different notation for $\t\psi(q)$ in (\ref{psi-int}) has also been given by Gradshteyn and Ryzhik
(see  \cite[p. 956, 8.370]{GR} and  \cite[p. 39, $(2.9)'$]{WZ}). Furthermore, let \begin{equation}\label{poly-ga-def}
\t\psi^{(n)}(q):=\left(\frac{d}{dq}\right)^n\t\psi(q),\quad n=0,1,2,\ldots.
\end{equation}
As in the classical situation (\ref{Classical}), we obtain the following representation
\begin{equation}\label{ga-poly}
\t\psi^{(n)}(q)=(-1)^{n+1}n!\zeta_E(n+1,q), \quad n = 0,1,2,\ldots
\end{equation}
(cf. \cite[p. 957, 8.374]{GR}).

Inspiring by the classical formula (\ref{Classical2}), for $q>0,$ we define the modified gamma function $\t\Gamma(q)$ from the differential 
equation $$\t\psi(q)=\frac{d}{dq}\log\t\Gamma(q).$$
Then $\t\Gamma(q)$ have the following infinite product which is an analogue of the Weierstrass--Hadamard product (\ref{Hadamard}),
$$\t\Gamma(q)=\frac1q e^{\t\gamma_0 q}\prod_{k=1}^\infty\left(e^{-\frac qk}\left(1+\frac qk\right)\right)^{(-1)^{k+1}},$$ 
where $\t\gamma_0$ is the modified Euler constant. (See Theorem \ref{thm-def}).

By the way, we obtain two new series expansions of $\pi$:
\begin{equation} 
\frac{\pi^2}{12}=\frac34-\sum_{k=1}^\infty(\zeta_E(2k+2)-1)
\end{equation}
(see (\ref{seriespi}))
and 
\begin{equation} \frac\pi2= \log2+2\sum_{k=1}^\infty\frac{(-1)^k}{k!}\t\ga_k\sum_{j=0}^kS(k,j)j!
\end{equation}
(see Proposition~\ref{log2-pi}).

 In Section \ref{Remarks}, we further provide several  interesting integrals which relate $\zeta_{E}(z,q)$, the gamma function, the beta function (also known as the first Eulerian integral) and the Gaussian hypergeometric
 functions. For example, for ${\rm Re}(\delta)>-1,{\rm Re}(\beta)>-1,q>0,|x|<q$ and $|v|<1,$ we obtain 
 \begin{equation*}\begin{aligned}
\int_0^1t^\beta(1-t)^\delta&(1-tv)^{-\alpha}\left[\zeta_E(z,q-xt)-\zeta_E(z,q)\right]dt \\
&=\sum_{j=1}^\infty\frac{\Gamma(z+j)}{\Gamma(z)j!}B(\beta+j+1,\delta+1) \\
&\quad\times {}_2F_1(\alpha,\beta+j+1;\delta+\beta+j+2;v)\zeta_E(z+j,q)x^j
\end{aligned}
\end{equation*}
(see Proposition \ref{hyp-int}).

\section{Preliminary lemmas}
In this section, to our purpose, firstly, we recall some facts about the digamma functions $\psi(q)$  and investigate some properties of   the modified digamma functions $\t\psi(q)$,
especially the relation between them and the difference equations. 

Substituting Taylor expansion (\ref{l-s-con}) into the following difference equation of $\zeta_{E}(z,q)$ (see \cite[Lemma 3.1]{KMS})
$$\begin{aligned}
(-1)^{n-1}\zeta_E(z,q+n)+\zeta_E(z,q)&=\sum_{j=0}^{n-1}(-1)^j(q+j)^{-z} \\
&=\sum_{j=0}^{n-1}(-1)^j\frac1{q+j}e^{-(z-1)\log(q+j)},
\end{aligned}$$
then expanding both sides as  the power series expansion in $z-1$ and comparing the coefficients, 
we have
\begin{equation}\label{ad-ga-g}
(-1)^{n-1}\t\ga_k(q+n)+\t\ga_k(q)=\sum_{j=0}^{n-1}(-1)^j(q+j)^{-1}\log^k(q+j)
\end{equation}
for $n=1,2,\ldots$ and $k=0,1,2\ldots.$
In particular, for $n=1$,
\begin{equation}\label{ad-ga}
\t\ga_k(q+1)+\t\ga_k(q)=q^{-1}\log^kq
\end{equation}
for $k\geq 0.$

From (\ref{Classical2}), the definition of  $\psi(q),$ and (\ref{Gamma}), Euler's definition of gamma functions,
we have
\begin{equation}\label{di-pro1}
\psi(q)=\int_0^\infty\left(\frac{e^{-t}}{t}- \frac{e^{-qt}}{1-e^{-t}}\right)dt,
\end{equation} (see \cite{AS,Co06,Co,GR,SC})
and 
\begin{equation}\label{di-pro2}
\psi(q)=-\ga+\int_0^\infty\frac{e^{-t}-e^{-qt}}{1-e^{-t}}dt,
\end{equation}
where $\ga$ is the Euler's constant (see \cite{AS,Co06,Co,De,GR,SC,ZW}).

Start with (\ref{psi-def})
\begin{equation*}
\t\psi(q):=-\t\ga_0(q)=-\zeta_E(1,q),
\end{equation*}
 by (\ref{di-pro2}) and the special case $s=1$ in (3.1) of \cite{WZ},
we obtain an integral representation for
the modified digamma function $\t\psi,$
\begin{equation}\label{psi-int}
\begin{aligned}
\t\psi(q)&=-\int_0^\infty\frac{e^{-qt}}{1+e^{-t}}dt =-\int_0^\infty\frac{e^{-qt}-e^{-(q+1)t}}{1-e^{-2t}}dt \\
&=-\frac12\int_0^\infty\frac{e^{-\frac q2(2t)}}{1-e^{-2t}}d(2t)+\frac12\int_0^\infty\frac{e^{-\frac {q+1}2(2t)}}{1-e^{-2t}}d(2t) \\
&=\frac12\left(\psi\left(\frac q2\right)-\psi\left(\frac{q+1}{2}\right)\right).
\end{aligned}
\end{equation}
Letting $q\to\frac q2$ in the following duplication formula of $\psi(q):$
$$\psi(2q)=\frac12\left(\psi(q)+\psi\left(q+\frac12\right)\right)+\log2,$$
we have 
\begin{equation}\label{eq-defi}
\frac12\left(\psi\left(\frac q2\right)-\psi\left(\frac{q+1}{2}\right)\right)=-\psi(q)+\psi(q/2)+\log2,
\end{equation}
then substituting into (\ref{psi-int}) we get (\ref{ps-ga}).

Letting $k=0$ in  (\ref{ad-ga}) and noticing (\ref{psi-def}), we get the  following difference equation satisfied by $\t\psi(q)$,
\begin{equation}\label{gen-psi1}\t\psi(q+1)+\t\psi(q)=-\frac1q,\end{equation}
and more generally, for $n\geq 1$,
\begin{equation}\label{gen-psi}
(-1)^{n-1}\t\psi(q+n)+\t\psi(q)=\sum_{k=1}^n\frac{(-1)^k}{q+k-1}.
\end{equation}

Then we state the following two lemmas. They will be used in the proofs of the main results.
\begin{lemma}[Boole summation formula]\label{BSF}
Let $\alpha,\beta$ and $l$ be integers such that $\alpha<\beta$ and $l\in\mathbb N.$ If $f^{(l)}(t)$ is absolutely integrable over $[\alpha,\beta].$
Then
$$\begin{aligned}
2\sum_{n=\alpha}^{\beta-1}(-1)^nf(n)& = \sum_{r=0}^{l-1}\frac{E_r(0)}{r!}\left((-1)^{\beta-1}f^{(r)}(\beta)+(-1)^{\alpha}f^{(r)}(\alpha) \right) \\
&\quad+
\frac1{(l-1)!}\int_{\alpha}^{\beta}\E_{l-1}(-t)f^{(l)}(t)dt,
\end{aligned}$$
where $\E_{n}(t),n=0,1,2,\ldots,$ is the $n$-th quasi-periodic Euler functions defended by
$\E_n(t+1)=-\E_n(t)$
for all $t,$ and $\E_n(t)=E_n(t)$ for $0\leq t < 1,$
where $E_n(t)$ denotes the $n$-th Euler polynomials.
\end{lemma}
\begin{remark} This summation formula  is obtained by Boole, but a similar one may be known by Euler as well (see \cite[Theorem 1.2]{CanDa}, \cite[Lemma 2.1]{KMS} and
\cite[24.17.1--2]{NIST}).
\end{remark}

\begin{lemma}\label{lem2-1}
\begin{itemize}
\item[(1)] For $0\leq\ell\leq n-1,$ we have
$$\sum_{k=0}^n\binom nk(-1)^kk^\ell=0,$$
in which we understand $0^0=1$ if $\ell=0,$ and $0^\ell=0$ otherwise.
\item[(2)] For $n\geq0,$ we have
$$\sum_{k=0}^n\binom nk(-1)^kk^n=(-1)^nn!.$$
\end{itemize}
\end{lemma}
\begin{proof} The results may be well-known, but for the reader’s convenience, we briefly review their proofs here.

The $n$-th order forward differences are given by (see \cite{Jor})
\begin{equation}\label{f-diff}
\Delta^nf(x)=\sum_{k=0}^n\binom nk(-1)^{n-k}f(x+k).
\end{equation}
On the other hand if $f(x)$ is a polynomial of degree $n,$ then
\begin{equation}\label{f-diff-0}
\Delta^{n+1}f(x)=0.
\end{equation}
So for $0\leq\ell\leq n-1$, setting $f(x)=x^\ell$ in (\ref{f-diff}), by  (\ref{f-diff-0}) and noticing that $0^0=1$ for convention,
we get Part (1) of the lemma.

To see Part (2), note that
$$\Delta^n(x^n)=n!.$$
Thus, Part (2) follows from (\ref{f-diff}).
\end{proof}

\section{Statement of new results}
In this section, we state our main results, that is, we shall present several  new integrals, series or infinite product representations  of the alternating Hurwitz zeta function
$\zeta_{E}(z,q),$
the alternating zeta function $\zeta_{E}(z)$, the  modified  Stieltjes constants $\tilde\gamma_k(q)$, the modified Euler constant $\tilde\gamma_{0}$, the modified  gamma function $\tilde\Gamma(q)$ and the modified digamma function $\tilde\psi(q)$. 
Their proofs will appear in the next section.
\begin{proposition}\label{pro0}
For $q>0,{\rm Re}(z)>-1$ and $\alpha=0,1,2\ldots,$ we have
$$\zeta_E(z,q)=\sum_{n=0}^\alpha\frac{(-1)^n}{(n+q)^{z}}-\frac12\frac{(-1)^\alpha}{(\alpha+q)^{z}}
-\frac12 z\int_\alpha^\infty\frac{\E_0(-t)}{(t+q)^{z+1}}dt,$$
where $\E_0(t)$ is the 0-th periodic Euler function defined by the Fourier series expansion
\begin{equation}\label{fu-eu0}
\E_0(t)=\frac4\pi\sum_{k=0}^\infty\frac{\sin(2k+1)\pi t}{2k+1}
\end{equation}
for $t$ not equal to any integer.
\end{proposition}
\begin{remark}
For the properties of the $k$-th quasi-periodic Euler functions $\E_{k}(x), k=0,1,2,\ldots,$ we refer to Section 1 of \cite{HKK2016}, especially equations (1.3)--(1.6) there.
\end{remark}

\begin{corollary}\label{cor}
We have
$$\zeta_E(z)=\frac12+\frac12 z\int_1^\infty\frac{\E_0(-t)}{t^{z+1}}dt,\quad{\rm Re}(z)>-1.$$
In particular,
\begin{equation}\label{Co(2)} \t\ga_0=\frac12+\frac12\sum_{j=1}^\infty(-1)^{j+1}\frac{1}{j(j+1)}=\frac12\left(\psi(2)-\psi\left(\frac32\right)+1\right).\end{equation}
\end{corollary}

\begin{remark}\label{rem-log}
Using the following three known formulas for digamma functions $\psi(q)$, $$\psi(q+1)=\psi(q)+\frac1q,$$ $$\psi(1)=-\ga~\textrm{and}~\psi\left(\frac12\right)=-2\log2-\ga,$$
(\ref{Co(2)}) in Corollary \ref{cor} implies that
\begin{equation}\label{remark14}
\begin{aligned}
\t\ga_0&=\frac12\left(\psi(2)-\psi\left(\frac32\right)+1\right)
=\frac12\left( \psi(1)-\psi\left(\frac12\right)\right)=\log 2.
\end{aligned}
\end{equation}
\end{remark}

\begin{corollary}\label{cor-hyper}
$$\begin{aligned}
\t\ga_1&=\frac12-\log2+2\sum_{k=0}^\infty\left(1+\frac12(\ga-2)\ga-\frac{\pi^2}{24} \right. \\
&\quad-\frac{(2k+1)^2\pi^2}{24}\,_3F_4\left(1,1,1;2,2,2,\frac52;-\frac{(2k+1)^2\pi^2}{4}\right) \\
&\quad\left.+\log(2k+1)\pi\big(\ga-1+\frac12\log(2k+1)\pi\big)\right),
\end{aligned}$$
where $\ga$ is the Euler constant.
\end{corollary}

\begin{theorem}\label{thm0}
Let $q>0$ and $\ell,\alpha=0,1,2\ldots.$ Then
$$\t\gamma_\ell(q)=\sum_{n=0}^\alpha(-1)^n\frac{\log^\ell(n+q)}{n+q}-(-1)^\alpha\frac{\log^\ell(\alpha+q)}{2(\alpha+q)}
+\frac12\int_\alpha^\infty\E_0(-t)g'_\ell(t)dt,$$
where $g_\ell(t)=\frac{\log^\ell(t+q)}{t+q}$ and $'$ denote differentiation with respect to $t.$
In particular, if $q=1,2,\ldots$ and $\ell=0,1,2,\ldots,$ then
$$\t\gamma_\ell(q)=\frac{\log^\ell q}{2q}+\frac{(-1)^q}{2}\sum_{j=q}^\infty(-1)^{j+1}
\left(\frac{\log^\ell (j+1)}{j+1}-\frac{\log^\ell j}{j}\right).$$
\end{theorem}
\begin{remark} Williams and Zhang \cite[Theorem 1]{ZW} proved a similar formula.
Theorem \ref{thm0} has the following advantage, it may be used to estimate the size of $\t\ga_\ell(q).$
\end{remark}

By letting $\alpha\to\infty $ in Theorem \ref{thm0}, we get the following limit formula.
\begin{corollary}\label{cor0}
$$\t\ga_\ell(q)=\lim_{\alpha\to\infty}\sum_{n=0}^\alpha(-1)^n\frac{\log^\ell(n+q)}{n+q}.$$
\end{corollary}

\begin{theorem}[Additional formula for $\t\gamma_\ell(q)$]\label{thm1}
Let $q>0$ and $|x|<q,$ and as usual, let $'$ denote differentiation with respect to the argument of a function, let
$^{(n)}$ denote the  $n$-fold differentiation.
Then
$$\t\gamma_\ell(q+x)=\t\gamma_\ell(q)+(-1)^\ell\sum_{j=2}^\infty\frac{x^{j-1}}{(j-1)!}
\sum_{k=0}^\ell\binom\ell k(-1)^k k!s(j,k+1)\zeta_E^{(\ell-k)}(j,q),$$
where $s(j,k)$ is the Stirling numbers of the first kind.
\end{theorem}

\begin{remark}
Letting $\ell=0,$ we have
\begin{equation}\label{ell-0}
\t\gamma_0(q+x)=\t\gamma_0(q)+\sum_{j=1}^\infty(-x)^j\zeta_E(j+1,q),
\end{equation}
because $s(j+1,1)=(-1)^jj!$ for $j=1,2,\ldots.$ It is known that  $\zeta_E(z,q)$ has the following integral representation, 
\begin{equation}\label{zeta-int}
\zeta_E(z,q)=\frac1{\Gamma(z)}\int_0^\infty\frac{e^{-(q-1)t}t^{z-1}}{e^t+1}dt,\quad\text{Re}(z)>0
\end{equation}
(see \cite[(3.1)]{WZ}).
If we put $z=j+1$ in (\ref{zeta-int}), then 
\begin{equation}\label{zeta-int2}
\zeta_E(j+1,q)=\frac1{\Gamma(j+1)}\int_0^\infty\frac{e^{-(q-1)t}t^{j}}{e^t+1}dt,\quad\text{Re}(z)>0.
\end{equation}
Substituting to (\ref{ell-0}) we get
\begin{equation}\label{add-con}
\begin{aligned}
\t\gamma_0(q+x)&=\t\gamma_0(q)+\sum_{j=1}^\infty\frac{(-x)^j}{j!}\int_0^\infty\frac{e^{-(q-1)t}t^j}{e^t+1}dt \\
&=\t\ga_0(q)+\int_0^\infty\frac{e^{-qt}(e^{-xt}-1)}{e^{-t}+1}dt \\
&=\t\ga_0(q)+\t\psi(q)-\t\psi(q+x) \\
&=\t\ga_0(q)-\t\ga_0(q)-\t\psi(q+x) \\
&=-\t\psi(q+x),
\end{aligned}
\end{equation}
in which, the third line follows from the integral representation for $\t\psi(q)$ (\ref{psi-int}) and the fourth line  follows from (\ref{psi-def}).
\end{remark}

\begin{remark}
By (\ref{l-s-con}) and (\ref{ga-poly}), we obtain 
$$\t\psi^{(n)}(q)=(-1)^{n+1}n!\sum_{k=0}^\infty\frac{(-1)^k}{k!}\t\ga_k(q)n^k,\quad n=1,2,\ldots.$$
This equation may be seen as an infinite linear system. Its solution  may express the modified Stieltjes coefficients in terms of
$\t\psi^{(n)}.$
By   (\ref{ga-poly}) and (\ref{zeta-int})  we have the following integral and series representations  for  $\t\psi^{(n)}(q)$
\begin{equation}\label{ga-poly-int}
(-1)^{n+1}\t\psi^{(n)}(q)=\int_0^\infty\frac{t^ne^{-qt}}{1+e^{-t}}dt=n!\sum_{k=0}^\infty\frac{(-1)^k}{(k+q)^{n+1}}
\end{equation}
for $q>0$ and $n=0,1,2,\ldots.$
\end{remark}

Equation (\ref{ell-0}) is reminiscent of  
Weierstrass's definition for the gamma functions (see e.g., \cite[p. 236]{WW}). This leads us to state the following theorem.

\begin{theorem}\label{thm-def}
If we define the modified gamma function $\t\Gamma(q)$ from the differential equation 
$$\t\psi(q)=\frac{d}{dq}\log\t\Gamma(q),$$
then it has the following infinite product representation 
\begin{equation}\label{th1} \t\Gamma(q)=\frac1q e^{\t\gamma_0 q}\prod_{k=1}^\infty\left(e^{-\frac qk}\left(1+\frac qk\right)\right)^{(-1)^{k+1}}.\end{equation}
Furthermore, we have
\begin{equation}\label{th2} \t\psi(q)=-\frac1q+\t\gamma_0+\sum_{k=1}^\infty(-1)^{k}\left(\frac1k-\frac1{k+q}\right).\end{equation}
\end{theorem}

As a corollary of the above result, we obtain the following series expansion of $\log\t\Gamma(q)$.
\begin{corollary}\label{cor-def}
$$\log\t\Gamma(q)=-\log q+\t\gamma_0 q+\sum_{k=2}^\infty\frac{(-1)^{k}\zeta_E(k)}{k}q^k,\quad|q|<1.$$
\end{corollary}

As an explicit representation coming from Theorem \ref{thm-def}, we have the following closed form.

\begin{corollary}\label{cor-def-1}
Let $q$ be a positive integer. Then
$$\t\psi(q)=-\frac1q+\t\gamma_0+\begin{cases}
\displaystyle\sum_{k=1}^{q}(-1)^k\frac1k&\text{if $q$ is even}, \\
\displaystyle\sum_{k=0}^{q-1}(-1)^{k}\frac1{k+1}-2\log2&\text{if $q$ is odd}.\\
\end{cases}$$
\end{corollary}

\begin{theorem}\label{thm2}
We have
$$\t\gamma_\ell^{(n)}(q)=(-1)^\ell\sum_{k=0}^\ell(-1)^kk!\binom\ell k s(n+1,k+1)\zeta_E^{(\ell-k)}(n+1,q).$$
\end{theorem}

If we restrict ourselves to the particular values $n=1,2,3,$ then using known values for Stirling
numbers of the first kind, we obtain the following result.

\begin{corollary}\label{cor1}
$$\t\gamma_\ell'(q)=(-1)^{\ell+1}(\zeta_E^{(\ell)}(2,q)+\ell \zeta_E^{(\ell-1)}(2,q)),$$
$$\t\gamma_\ell''(q)=(-1)^{\ell}(2\zeta_E^{(\ell)}(3,q)+3\ell \zeta_E^{(\ell-1)}(3,q)+\ell(\ell-1)\zeta_E^{(\ell-2)}(3,q)),$$
$$\begin{aligned}
\t\gamma_\ell'''(q)&=(-1)^{\ell+1}(6\zeta_E^{(\ell)}(4,q)+11\ell \zeta_E^{(\ell-1)}(4,q)+6\ell(\ell-1)\zeta_E^{(\ell-2)}(4,q) \\
&\quad+\ell(\ell-1)(\ell-2)\zeta_E^{(\ell-3)}(4,q)).
\end{aligned}$$
\end{corollary}

\begin{theorem}[Asymptotic relation]\label{thm3}
Let $E_n(x)$ denote the $n$-th Euler polynomial. We have the asymptotic formula
$$\begin{aligned}
\t\gamma_\ell(q)&\sim\frac{\log^\ell q}{2q} -\frac{\ell\log^{\ell-1}q-\log^\ell q}{4q^2} \\
&\quad+\sum_{k=1}^\infty\frac{E_{2k+1}(0)}{2q^{2k+2}(2k+1)!}\sum_{j=0}^\ell\binom\ell j j!
s(2k+2,j+1)\log^{\ell-j}q
\end{aligned}$$
as $q\to\infty$.
\end{theorem}

\begin{lemma}\label{lem1}
\begin{itemize}
\item[(1)] For $q>2,$ we have
$$\zeta_E(z,q)=\frac{(q-2)^{1-z}-(q-1)^{1-z}}{2(z-1)}-\frac1{\Gamma(z)}\sum_{k=1}^\infty\frac{2^k\Gamma(z+k)}{(k+1)!}\zeta_E(z+k,q).$$
\item[(2)] For $q>1,$ we have
$$\zeta_E(z,q)=\frac{(q-1)^{1-z}-q^{1-z}}{2(z-1)}-\frac1{\Gamma(z)}\sum_{k=1}^\infty\frac{\Gamma(z+2k)}{(2k+1)!}\zeta_E(z+2k,q).$$
\end{itemize}
\end{lemma}

As a result, we have
\begin{theorem}\label{thm4}
\begin{itemize}
\item[(1)] For $q>2,$ we have
$$\begin{aligned}
\t\ga_\ell(q)&=\frac{\log^{\ell+1}(q-1)-\log^{\ell+1}(q-2)}{2(\ell+1)} \\
&\quad-\ell!\sum_{k=1}^\infty\frac{(-1)^k2^k}{(k+1)!}\sum_{j=0}^\ell\frac{(-1)^j}{j!}s(k+1,\ell-j+1)\zeta_E^{(j)}(k+1,q).
\end{aligned}$$
\item[(2)] For $q>1,$ we have
$$\begin{aligned}
\t\ga_\ell(q)&=\frac{\log^{\ell+1}q-\log^{\ell+1}(q-1)}{2(\ell+1)} \\
&\quad-\ell!\sum_{k=1}^\infty\frac{1}{(2k+1)!}\sum_{j=0}^\ell\frac{(-1)^j}{j!}s(2k+1,\ell-j+1)\zeta_E^{(j)}(2k+1,q).
\end{aligned}$$
\end{itemize}
\end{theorem}

By (\ref{ga-poly}) and Lemma \ref{lem1}
we immediately get the following series representations  of  $\t\psi^{(n)}(q)$.
\begin{proposition}\label{cor2}
\begin{itemize}
\item[(1)] For $q>2,$ we have
$$\begin{aligned}
\t\psi^{(n)}(q)&=\frac12(-1)^{n+1}(n-1)!\left(\frac1{(q-2)^n}-\frac1{(q-1)^n}\right) \\
&\quad+(-1)^n\sum_{k=1}^\infty
\frac{2^k(n+k)!}{(k+1)!}\zeta_E(n+k+1,q).
\end{aligned}$$
\item[(2)] For $q>1,$ we have
$$\begin{aligned}
\t\psi^{(n)}(q)&=\frac12(-1)^{n+1}(n-1)!\left(\frac1{(q-1)^n}-\frac1{q^n}\right) \\
&\quad+(-1)^n\sum_{k=1}^\infty
\frac{(n+2k)!}{(2k+1)!}\zeta_E(n+2k+1,q).
\end{aligned}$$
\end{itemize}
\end{proposition}

\begin{remark}
(1) In particular, setting $n=1$ in Proposition \ref{cor2}(2) and integrating the result equation, we get a series representation of
the modified digamma function $\t\psi.$ In particular, by applying  the fundamental theorem of calculus,
from (\ref{J-2}) and Proposition \ref{cor2}(2) we have
\begin{equation}\label{remark1} \t\psi(q)=\frac12\log\left(1-\frac1q\right)+\sum_{k=1}^\infty\frac{\zeta_E(2k+1,q)}{2k+1}, \quad q>1.\end{equation}
By (\ref{psi-def}) we can also obtain a series representation for the modified Euler constant $\t\ga_0.$ That is, setting $q=2$ in (\ref{remark1}), we get
\begin{equation}\label{remark2} \t\psi(2)=-\frac12\log 2+\sum_{k=1}^\infty\frac{\zeta_E(2k+1,2)}{2k+1}.\end{equation}
Then from \begin{equation}\label{remark121-1} \t\psi(2)+\t\psi(1)=-1\end{equation} (see (\ref{gen-psi1})), \begin{equation}\label{remark121-2} \zeta_E(2k+1,2)=-\zeta_E(2k+1,1)+1=-\zeta_E(2k+1)+1\end{equation}
(see (\ref{J-1}) for the difference equation of $\zeta_{E}(z,q)$), 
and equations (\ref{gammak}), (\ref{psi-def}), (\ref{remark2}), (\ref{remark121-1}) and (\ref{remark121-2}), we obtain
\begin{equation}\begin{aligned} 
\t\ga_0&=-\t\psi(1) \\
&=1+\t\psi(2)\\&=1-\frac12\log 2-\sum_{k=1}^\infty\frac{\zeta_E(2k+1)-1}{2k+1}.
\end{aligned}\end{equation}
A similar equation was first obtained by Stieltjes \cite{St} for the Riemann zeta function $\zeta(z)$ using a different method.

(2) Since (see (\ref{J-1}))
$$\zeta_E(2k+2,2)+\zeta_E(2k+2,1)=1,$$ 
we have 
\begin{equation}\label{remark3} \zeta_E(2,2)=-\zeta_E(2,1)+1=-\zeta_E(2)+1.\end{equation}
Letting $n=1$ in Proposition \ref{cor2}(2), we have 
\begin{equation}\label{remark4} \t\psi'(2)=\frac14+\sum_{k=1}^\infty(\zeta_E(2k+2)-1).\end{equation} 
Since (see \cite[p. 808, 23.2.26]{AS})
 $$\zeta_E(2)=\eta(2)=\sum_{n=1}^{\infty}\frac{(-1)^{n+1}}{n^{2}}=\frac{\pi^2}{12},$$ 
by (\ref{ga-poly}) and (\ref{remark3}) we get
\begin{equation}\label{remark5}\begin{aligned} \t\psi'(2)&=\zeta_{E}(2,2)\\&=-\zeta_E(2)+1\\&=-\frac{\pi^2}{12}+1.\end{aligned}\end{equation}
Then comparing (\ref{remark4}) and (\ref{remark5}), we obtain a series expansion of $\pi$ by $\zeta_{E}(z,q),$
\begin{equation}\label{seriespi}
\frac{\pi^2}{12}=\frac34-\sum_{k=1}^\infty(\zeta_E(2k+2)-1).
\end{equation}
\end{remark}

\begin{proposition}\label{pro2}
\begin{itemize}
\item[(1)] For {\rm Re}$(s)>0,$ we have
$$\sum_{k=j}^\infty\frac{s^k}{(k-j)!}\int_0^2\t\ga_k(q)dq=0.$$
\item[(2)] For $n=0,1,2,\ldots,$ we have
$$\sum_{k=0}^\infty\frac{\t\ga_{k+n}(q)}{k!}=(-1)^n\zeta_E^{(n)}(0,q).$$
\item[(3)] For {\rm Re}$(q)>0,$ we have
$$\sum_{k=0}^\infty\frac{\t\ga_{k+1}(q)}{k!}=
\log\Gamma\left(\frac{1+q}2\right)-\log\Gamma\left(\frac q2\right)+\frac12\log 2.$$
\end{itemize}
\end{proposition}

\begin{proposition}\label{pro3}
For integers $j\geq1$ we have
$$\frac{(-1)^j}{j!}\frac{d}{dq}\t\ga_j(q)=-\sum_{k=j-1}^\infty\frac{(-1)^k}{k!}\binom{k+1}j\t\ga_k(q)$$
and
$$-\t\psi'(q)=-\sum_{k=0}^\infty\frac{(-1)^k}{k!}\t\ga_k(q),$$
where $\t\psi$ is the modified digamma function.
\end{proposition}

\section{Proofs of the main results}\label{pfs}
In this section, we prove theorems, propositions and corollaries which have been  stated in the previous section.

\subsection*{Proof of Proposition \ref{pro0}}
Taking $f(x)=\frac1{(x+q)^z}$ with $q>0$ and Re$(z)>0,$ setting $l=1$ and letting $\beta\to\infty$ in Lemma \ref{BSF}, we obtain
$$
\sum_{n=\alpha}^\infty \frac{(-1)^n}{(n+q)^z}=\frac12(-1)^\alpha\frac1{(\alpha+q)^z}
-\frac12z\int_\alpha^\infty\frac{\E_0(-t)}{(t+q)^{z+1}}dt.
$$
Then from the definition of $\zeta_{E}(z,q)$ (see (\ref{E-zeta-def})), we have
$$
\begin{aligned}
\zeta_E(z,q)&=\sum_{n=0}^{\alpha-1}\frac{(-1)^n}{(n+q)^z}
+\frac12\frac{(-1)^\alpha}{(\alpha+q)^z}
-\frac12z\int_\alpha^\infty\frac{\E_0(-t)}{(t+q)^{z+1}}dt \\
&=\sum_{n=0}^{\alpha}\frac{(-1)^n}{(n+q)^z}
-\frac12\frac{(-1)^\alpha}{(\alpha+q)^z}
-\frac12z\int_\alpha^\infty\frac{\E_0(-t)}{(t+q)^{z+1}}dt,
\end{aligned}$$
which completes our proof.

\subsection*{Proof of Corollary \ref{cor}}

Setting $\alpha=0$ in Proposition \ref{pro0}, we get
\begin{equation}\label{cor-1}
\zeta_E(z,q)=\frac1{2q^z}-\frac12 z\int_0^\infty\frac{\E_0(-t)}{(t+q)^{z+1}}dt,\quad{\rm Re}(z)>-1.
\end{equation}
It is readily checked that the quasi-periodic Euler function $\E_0(t)$ satisfies  the following identity
\begin{equation}\label{cor-1-0}
\E_0(-t)=\E_0(-t-1+1)=-\E_0(-t-1)
\end{equation}
(see \cite[(1.6)]{HKK}).
Then by (\ref{cor-1}) with $q=1,$ and (\ref{cor-1-0}), we find that
\begin{equation}\label{cor-2}
\zeta_E(z)=\zeta_{E}(z,1)=\frac1{2}+\frac12 z\int_1^\infty\frac{\E_0(-t)}{t^{z+1}}dt,
\end{equation}
Notice that  by  the quasi-periodic property of the function $\E_{0}(t)$ (see \cite[(1.6)]{HKK2016}), for $j\leq t<j+1$ we have
$$\E_0(-t)=(-1)^{j+1},$$
then from (\ref{gammak}), (\ref{gamma0}) and  (\ref{cor-2}) with $z=1,$  we get
\begin{equation}\label{cor-3}
\begin{aligned}
\t\ga_0&=\t\gamma_0(1)=\zeta_{E}(1,1)\\&=\frac1{2}+\frac12 \int_1^\infty\frac{\E_0(-t)}{t^{2}}dt \\
&=\frac1{2}+\frac12 \sum_{j=1}^\infty\int_j^{j+1}\frac{\E_0(-t)}{t^{2}}dt \\
&=\frac1{2}+\frac12 \sum_{j=1}^\infty(-1)^{j+1}\int_j^{j+1}\frac{1}{t^{2}}dt \\
&=\frac1{2}+\frac12 \sum_{j=1}^\infty(-1)^{j+1}\frac1{j(j+1)} .
\end{aligned}
\end{equation}
From an identity \cite[p. 255, (33)]{SC} with $a=1,$ we have
$$\sum_{j=1}^\infty(-1)^{j+1}\frac1{j(j+1)} =-\psi\left(1+\frac12\right)+\psi(2),$$
substituting into (\ref{cor-3}) yields
$$\t\ga_0=\frac12\left(\psi(2)-\psi\left(\frac32\right)+1\right),$$
which completes our proof.

\subsection*{Proof of Corollary \ref{cor-hyper}}
By (\ref{gammak}) and  the Taylor expansion (\ref{l-s-con}), we have \begin{equation}\label{corollary1.5-1}\t\gamma_1=\t\gamma_1(1)=-\zeta_E'(1,1)=-\zeta_E'(1).\end{equation}
Deriving  the identity (\ref{cor-2}) with respect to $z$ we obtain
\begin{equation}\label{cor-h-1}
\zeta_E'(z)=\frac12 \int_1^\infty\frac{\E_0(-t)}{t^{z+1}}dt-\frac12z \int_1^\infty\frac{\E_0(-t)}{t^{z+1}}\log t\,dt,
\end{equation}
and  taking $z=1$ in (\ref{cor-h-1}), we get
\begin{equation}\label{corollary1.5-2}
\zeta_E'(1)=\frac12 \int_1^\infty\frac{\E_0(-t)}{t^{2}}dt-\frac12 \int_1^\infty\frac{\E_0(-t)}{t^{2}}\log t\,dt.
\end{equation}
By the second line of (\ref{cor-3}) and (\ref{remark14}) in Remark \ref{rem-log}, we have
\begin{equation}\label{corollary1.5-3}\frac12 \int_1^\infty\frac{\E_0(-t)}{t^{2}}dt=\t\ga_{0}-\frac12=\log 2-\frac12.\end{equation}
Then combing (\ref{corollary1.5-1}), (\ref{corollary1.5-2}) and (\ref{corollary1.5-3}), we see that
\begin{equation}\label{cor-h-2}
\t\ga_1=\frac12-\log2+\frac12 \int_1^\infty\frac{\E_0(-t)}{t^{2}}\log t\,dt.
\end{equation}
Substituting  the Fourier expansion of $E_{0}(t)$ (see (\ref{fu-eu0})) into the above equation, we have 
\begin{equation}\label{corollary1.5-4}
\t\ga_1=\frac12-\log2+2\sum_{k=0}^\infty\frac1{(2k+1)\pi} \int_1^\infty\frac{\sin(2k+1)\pi t}{t^{2}}\log t\,dt.
\end{equation}
Finally, by taking $\kappa=(2k+1)\pi$ in the known integral (see \cite[(2.7)]{Co13})
\begin{equation}\label{cor-h-3}
\begin{aligned}
\frac1{\kappa}\int_1^\infty\frac{\sin\kappa t}{t^2}\log t\,dt
&=1+\frac12(\ga-2)\ga-\frac{\pi^2}{24}  \\
&\quad-\frac{\kappa^2}{24}\,_3F_4\left(1,1,1;2,2,2,\frac52;-\frac{\kappa^2}{4}\right) \\
&\quad+\log\kappa\left(\ga-1+\frac12\log\kappa\right).
\end{aligned}
\end{equation}
and substituting into (\ref{corollary1.5-4}), we get the corollary.

\subsection*{Proof of Theorem \ref{thm0}}

We first expand each term on the right hand side of Proposition \ref{pro0} as power series in $z-1$:
\begin{equation}\label{thm0-1}
\begin{aligned}
\frac{(-1)^n}{(n+q)^z}&=\frac{(-1)^n}{n+q}\sum_{k=0}^\infty\frac{(-1)^k\log^k(n+q)}{k!}(z-1)^k, \\
\frac12\frac{(-1)^\alpha}{(\alpha+q)^z}&=\frac{(-1)^\alpha}{2(\alpha+q)}\sum_{k=0}^\infty\frac{(-1)^k\log^k(\alpha+q)}{k!}(z-1)^k,
\end{aligned}
\end{equation}
thus Proposition \ref{pro0} becomes 
\begin{equation}\label{thm0-3}
\begin{aligned}
\zeta_E(z,q)&=\sum_{k=0}^\infty\left(\sum_{n=0}^\alpha\frac{(-1)^{n+k}\log^k(n+q)}{n+q}
-\frac12 \frac{(-1)^{\alpha+k}\log^k(\alpha+q)}{\alpha+q}\right) \frac{(z-1)^k}{k!} \\
&\quad-\frac12 z\int_\alpha^\infty\frac{\E_0(-t)}{(t+q)^{z+1}}dt.
\end{aligned}
\end{equation}
Then substituting (\ref{l-s-con}) into the the left hand side of the above equation, we have
\begin{equation}\label{l-s-con-1}
\begin{aligned}
&\sum_{k=0}^\infty\frac{(-1)^k\t\ga_k(q)}{k!}(z-1)^k \\
=&\sum_{k=0}^\infty\left(\sum_{n=0}^\alpha\frac{(-1)^{n+k}\log^k(n+q)}{n+q}
-\frac12 \frac{(-1)^{\alpha+k}\log^k(\alpha+q)}{\alpha+q}\right) \frac{(z-1)^k}{k!} \\
&-\frac12 z\int_\alpha^\infty\frac{\E_0(-t)}{(t+q)^{z+1}}dt.
\end{aligned}
\end{equation}
Set
$$f(z)=z\int_\alpha^\infty\frac{\E_0(-t)}{(t+q)^{z+1}}dt,$$
then
\begin{equation}\label{thm0-2}
\begin{aligned}
f^{(k)}(z)&=(-1)^k\int_\alpha^\infty\frac{\E_0(-t)(z\log^k(t+q)-k\log^{k-1}(t+q))}{(t+q)^{z+1}}dt, \\
f^{(k)}(1)&=(-1)^{k-1}\int_\alpha^\infty \E_0(-t)g'_k(t) dt,
\end{aligned}
\end{equation}
where $g_k(t)=\frac{\log^k(t+q)}{t+q}$. Thus  
differentiating  both sides of (\ref{l-s-con-1}) $\ell$ times with respect to $z,$ then setting $z=1$ in the result equation, the first part now follows from (\ref{thm0-2}).

To see the second part,
for $q=1, 2, 3,\ldots$ and $\ell=0,1,2,\ldots,$ by the first part with $\alpha=0,$ we obtain
\begin{equation}\label{thm0-4}
\begin{aligned}
\t\gamma_\ell(q)&=\frac{\log^\ell q}{2q}+\frac12\int_0^\infty\E_0(-t)\frac{\log^{\ell-1}(t+q)}{(t+q)^2}(\ell-\log(t+q))dt \\
&=\frac{\log^\ell q}{2q}+\frac12(-1)^q\int_q^\infty\E_0(-y)\frac{\log^{\ell-1}y}{y^2}(\ell-\log y)dy
\end{aligned}
\end{equation}
with a change of variable $y=t+q,$ where we have used the following identity  
\begin{equation}\label{qu-Eu-q}
\E_0(-t)=\E_0(-t-q+q)=(-1)^q\E_0(-t-q).
\end{equation}
Moreover, using the integration by parts (or simply good integral tables) we find that for $\ell=0,1,2,\ldots,$
\begin{equation}\label{p-int}
\int \frac{\log^{\ell-1}y}{y^2}(\log y-\ell)dy=-\frac{\log^\ell y}{y}+C.
\end{equation}
In proceeding  as the proof of Corollary \ref{cor}, we get that
\begin{equation}\label{thm0-4-1}
\begin{aligned}
\t\gamma_\ell(q)
&=\frac{\log^\ell q}{2q}+\frac12(-1)^q\sum_{j=q}^\infty(-1)^{j+1}\int_j^{j+1}\frac{\log^{\ell-1}y}{y^2}(\ell-\log y)dy \\
&=\frac{\log^\ell q}{2q}+\frac12(-1)^q\sum_{j=q}^\infty(-1)^{j+1}\left(\frac{\log^{\ell}(j+1)}{j+1}-\frac{\log^\ell j}{j}\right),
\end{aligned}
\end{equation}
in which, we have used (\ref{p-int}).
This shows the second part.

\subsection*{Proof of Theorem \ref{thm1}}

The Wilton-type formula for the alternating Hurwitz zeta function $\zeta_{E}(z,q)$
is
\begin{equation}\label{w-form}
\zeta_E(z,q+x)=\zeta_E(z,q)+\sum_{j=1}^\infty\frac{(-1)^j}{j!}\frac{\Gamma(z+j)}{\Gamma(z)}\zeta_E(z+j,q)x^j,
\end{equation}
where $q>0$ with $|x|<q.$ (See \cite[Lemma 3.2(1)]{KMS}).

So (\ref{l-s-con}) yields
\begin{equation}\label{s-1}
\begin{aligned}
\sum_{k=0}^\infty(-1)^k\frac{\t\ga_k(q+x)}{k!}(z-1)^k&=\sum_{k=0}^\infty(-1)^k\frac{\t\ga_k(q)}{k!}(z-1)^k \\
&\quad+\sum_{j=1}^\infty\frac{(-1)^j}{j!}(z)_j\zeta_E(z+j,q)x^j,
\end{aligned}
\end{equation}
where
\begin{equation}\label{po-s}
(z)_j=z(z+1)\cdots(z+j-1)=\frac{\Gamma(z+j)}{\Gamma(z)}
\end{equation}
denotes the Pochhammer's symbol (the rising factorial function).

It may be interesting to remark here that there is an alternative proof of (\ref{w-form}) which is based on the Mellin transform
\begin{equation}\label{m-tran}
\Gamma(z)\zeta_E(z,q)=\int_0^\infty\frac{e^{-(q-1)t}t^{z-1}}{e^t+1}dt,\quad\text{Re}(z)>0.
\end{equation}
In fact, by the Taylor expansion of $e^{-xt}$ and the Lebesque's dominated convergence theorem we  
obtain
\begin{equation}
\begin{aligned}
\zeta_E(z,q+x)&=\frac{1}{\Gamma(z)}\int_0^\infty\frac{e^{-xt}e^{-(q-1)t}t^{z-1}}{e^t+1}dt\\&=\sum_{j=0}^\infty(-1)^j\frac{x^j}{\Gamma(z)j!}\int_0^\infty\frac{e^{-(q-1)t}t^{z+j-1}}{e^t+1}dt,
\end{aligned}
\end{equation}
then (\ref{w-form}) follows from (\ref{m-tran}).

Now we come back to our proof. 
 Let $s(j,k)$ be the Stirling numbers of the first kind. Since  
$$\left.\left(\frac{d}{dz}\right)^k(z)_j\right|_{z=1}=(-1)^{k+j}k!s(j+1,k+1)$$
(see \cite[p. 452, Lemma 1]{Co}),
by Leibniz's rule we have
\begin{equation}\label{d-1}
\begin{aligned}
\left.\left(\frac{d}{dz}\right)^\ell(z)_j\zeta_E(z+j,q)\right|_{z=1}
&=\left.\sum_{k=0}^\ell\binom \ell{k}\left(\left(\frac{d}{dz}\right)^k(z)_j\right)\zeta_E^{(\ell-k)}(z+j,q)\right|_{z=1}\\
&=\sum_{k=0}^\ell\binom\ell k(-1)^{k+j} k!s(j+1,k+1)\zeta_E^{(\ell-k)}(j,q).
\end{aligned}
\end{equation}
Then by taking $\ell$-th derivatives of (\ref{s-1}) with respect to $z,$ and setting $z=1$ in the result equation, we get
\begin{equation}(-1)^\ell\t\gamma_\ell(q+x)=(-1)^\ell\t\gamma_\ell(q)+\sum_{j=2}^\infty\frac{x^{j-1}}{(j-1)!}\sum_{k=0}^\ell\binom\ell k(-1)^k k!s(j,k+1)\zeta_E^{(\ell-k)}(j,q),
\end{equation}
which completes the proof of Theorem \ref{thm1}.

\subsection*{Proof of Theorem \ref{thm-def}}

An interesting special case of (\ref{ell-0}) is obtained by setting $q=1$:
\begin{equation}\label{d-q=1}
\t\gamma_0(1+x)=\t\gamma_0+\sum_{j=1}^\infty(-1)^j\zeta_E(j+1)x^j,\quad|x|<1.
\end{equation}
Then applying  (\ref{psi-def}) to the left hand side, we obtain the identity
\begin{equation}\label{d-q=1-1}
\t\psi(1+x)=-\t\gamma_0+\sum_{j=1}^\infty(-1)^{j+1}\zeta_E(j+1)x^j,
\end{equation}
where $|x|<1.$ 
Since $$\zeta_{E}(z)=\sum_{n=1}^{\infty}\frac{(-1)^{n+1}}{n^{z}},$$ 
we have
\begin{equation}\label{thm-def-2}
\begin{aligned}
\sum_{j=1}^\infty(-1)^{j+1}\zeta_E(j+1)x^j&=\sum_{j=1}^\infty(-1)^{j+1}x^j\sum_{k=1}^\infty\frac{(-1)^{k+1}}{k^{j+1}} \\
&=\sum_{k=1}^\infty(-1)^{k+1}\frac x{k^2}\sum_{j=0}^\infty(-1)^{j}\frac{x^j}{k^{j}} \\
&=\sum_{k=1}^\infty(-1)^{k+1}\frac{x}{k^2\left(1+\frac xk\right)} \\
&=\sum_{k=1}^\infty(-1)^{k+1}\left(\frac1k-\frac1{k+x}\right),
\end{aligned}
\end{equation}
thus from (\ref{gen-psi1}) and (\ref{d-q=1-1})
\begin{equation}\label{d-q=1-2}
\begin{aligned}
-\t\psi(x)-\frac1x&=\t\psi(1+x)\\
&=-\t\gamma_0+\sum_{j=1}^\infty(-1)^{j+1}\zeta_E(j+1)x^j
\\&=-\t\gamma_0+\sum_{k=1}^\infty(-1)^{k+1}\left(\frac1k-\frac1{k+x}\right).
\end{aligned}
\end{equation}
Since $$\t\psi(q)=\frac{d}{dq}\log\t\Gamma(q),$$ integrating both sides of  (\ref{d-q=1-2}) we get
\begin{equation}\label{lgamma} 
\log\t\Gamma(q)=-\log q+\t\gamma_0 q+\sum_{k=1}^\infty(-1)^{k}\left(\frac qk-\log\left(1+\frac qk\right)\right)
\end{equation}
and $$\t\Gamma(q)=\frac1q e^{\t\gamma_0 q}\prod_{k=1}^\infty\left(e^{-\frac qk}\left(1+\frac qk\right)\right)^{(-1)^{k+1}}.$$
Letting $x\to q$ in (\ref{d-q=1-2}), we have
$$\t\psi(q)=-\frac1q+\t\gamma_0+\sum_{k=1}^\infty(-1)^{k}\left(\frac1k-\frac1{k+q}\right).$$
These are our results.

\subsection*{Proof of Corollary \ref{cor-def}}

By (\ref{lgamma}) we have
  \begin{equation}\label{cor-def-pf}
\begin{aligned}
\log\t\Gamma(q)&=-\log q+\t\gamma_0 q+\sum_{n=1}^\infty(-1)^{n+1}\left(\log\left(1+\frac qn\right)-\frac qn\right) \\
&=-\log q+\t\gamma_0 q+\sum_{n=1}^\infty(-1)^{n+1}\sum_{k=2}^\infty\frac{(-1)^{k-1}}{k}\left(\frac qn\right)^k \\
&=-\log q+\t\gamma_0 q+\sum_{k=2}^\infty\frac{(-1)^{k-1}\zeta_E(k)}{k}q^k,
\end{aligned}
\end{equation}
because $$\zeta_{E}(z)=\sum_{n=1}^{\infty}\frac{(-1)^{n+1}}{n^{z}}.$$ 
Now Corollary \ref{cor-def} follows.

\subsection*{Proof of Corollary \ref{cor-def-1}}

To obtain a closed form expression  for the modified digamma function $\t\psi(q)$ from Theorem \ref{thm-def}, we shall work on the infinite series  $$\sum_{k=1}^\infty(-1)^{k}\left(\frac1k-\frac1{k+q}\right).$$ 
For $q$ a positive even integer, if $N$ is big enough, then
$$\sum_{k=1}^N(-1)^{k}\left(\frac1k-\frac1{k+q}\right) 
=\sum_{k=1}^q(-1)^k\frac1k+\sum_{k=1}^{q}(-1)^{N+k-1}\frac{1}{N+k}$$
and 
\begin{equation}\label{cor-def-pf-1}
\begin{aligned}
\sum_{k=1}^\infty(-1)^{k}\left(\frac1k-\frac1{k+q}\right)
&=\lim_{N\to\infty}\sum_{k=1}^N(-1)^{k}\left(\frac1k-\frac1{k+q}\right) \\
&=\sum_{k=1}^q(-1)^k\frac1k.
\end{aligned}
\end{equation}
For  $q\geq 3$ is a positive odd integer, if $N$ is big enough, then
$$\sum_{k=1}^N(-1)^{k}\left(\frac1{k+1}-\frac1{k+q}\right) 
=\sum_{k=1}^{q-1}(-1)^k\frac1{k+1}+\sum_{k=1}^{q-1}(-1)^{N+k-1}\frac{1}{N+k+1}.$$
In addition, $$\log2=\sum_{n=1}^{\infty}\frac{(-1)^{n+1}}{n}$$
and $$\sum_{k=1}^\infty(-1)^{k}\left(\frac1k-\frac1{k+1}\right)=1-2\log2.$$  
Thus, for $q$ is a positive odd integer, we have
\begin{equation}\label{cor-def-pf-2}
\begin{aligned}
\sum_{k=1}^\infty(-1)^{k}\left(\frac1k-\frac1{k+q}\right)
&=\sum_{k=1}^\infty(-1)^{k}\left(\left(\frac1k-\frac1{k+1}\right)+\left(\frac1{k+1}-\frac1{k+q}\right)\right) \\
&=\sum_{k=1}^\infty(-1)^{k}\left(\frac1k-\frac1{k+1}\right) \\
&\quad+\lim_{N\to\infty}\sum_{k=1}^N(-1)^{k}\left(\frac1{k+1}-\frac1{k+q}\right) \\
&=1-2\log2+\sum_{k=1}^{q-1}(-1)^k\frac1{k+1} \\
&=\sum_{k=0}^{q-1}(-1)^k\frac1{k+1}-2\log2.
\end{aligned}
\end{equation}
Substituting (\ref{cor-def-pf-1}) and (\ref{cor-def-pf-2}) into (\ref{th2}) respectively, we get our result.

\subsection*{Proof of Theorem \ref{thm2}}
From the differential calculus, the $n$-fold differentiation of $\t\gamma_\ell(q)$ with respect to $q$ can be expressed by the equation
\begin{equation}\label{nth-diff}
\t\gamma_\ell^{(n)}(q)=\lim_{x\to0}\frac{1}{x^n}\sum_{i=0}^n\binom ni(-1)^{n-i}\t\gamma_\ell(q+ix).
\end{equation}
By Theorem \ref{thm1}, we have 
$$\t\gamma_\ell(q+ix)=\t\gamma_\ell(q)+(-1)^\ell\sum_{j=2}^\infty\frac{(ix)^{j-1}}{(j-1)!}
\sum_{k=0}^\ell\binom\ell k(-1)^k k!s(j,k+1)\zeta_E^{(\ell-k)}(j,q),$$
where $s(j,k)$ is the Stirling numbers of the first kind.

For simplification of the notation, we denote by
\begin{equation}\label{Ajq} A_\ell(j,q)=\sum_{k=0}^\ell\binom\ell k(-1)^k k!s(j,k+1)\zeta_E^{(\ell-k)}(j,q).\end{equation}
Then (\ref{nth-diff}) becomes  
\begin{equation}\label{nth-diff-1}
\begin{aligned}
\t\gamma_\ell^{(n)}(q)&=\lim_{x\to0}\frac{(-1)^n}{x^n}\left(\sum_{i=0}^n\binom ni(-1)^{i}\t\gamma_\ell(q)\right. \\
&\quad\left.+(-1)^\ell\sum_{i=1}^n\binom ni(-1)^{i}
\sum_{j=2}^\infty\frac{(ix)^{j-1}}{(j-1)!}A_\ell(j,q)\right) \\
&=\lim_{x\to0}\frac{(-1)^n}{x^n}\left((-1)^\ell\sum_{j=2}^\infty\frac{x^{j-1}}{(j-1)!}\left(\sum_{i=1}^n\binom ni(-1)^{i}i^{j-1}\right)A_\ell(j,q)\right) \\
&=\lim_{x\to0}\frac{(-1)^n}{x^n}\left((-1)^\ell\frac{x^n}{n!}\left(\sum_{i=1}^n\binom ni(-1)^{i}i^{n}\right)A_\ell(n+1,q) \right.\\
&\quad+\left.(-1)^\ell\sum_{j=n+2}^\infty\frac{x^{j-1}}{(j-1)!}\left(\sum_{i=1}^n\binom ni(-1)^{i}i^{j-1}\right)A_\ell(j,q)\right)  \\
&=(-1)^\ell A_\ell(n+1,q),
\end{aligned}
\end{equation}
in which, the second and third equalities follow from Lemma \ref{lem2-1}(1),
and the last equation follows from  Lemma \ref{lem2-1}(2).
Finally, substituting the expression (\ref{Ajq}) into the last line of (\ref{nth-diff-1}) we get our result.

\subsection*{Proof of  Theorem \ref{thm3}}

In order to obtain Theorem \ref{thm3}, we apply an asymptotic expansion of $\zeta_{E}(z,q)$ obtained by the authors recently,
\begin{equation}\label{thm-p1}
\zeta_E(z,q)=\frac12q^{-z}+\frac14zq^{-z-1}-\frac12\sum_{k=2}^\infty\frac{E_k(0)}{k!}(z)_kq^{-z-k}
\end{equation}
as $q\to\infty,$ where $(z)_k$ is the function  in (\ref{po-s}). (See \cite[Theorem 2.3]{HKKS}).
Note that
\begin{equation}\label{thm-p2}
q^{-z}=\frac1qe^{-(z-1)\log q}=\sum_{k=0}^\infty\frac{(-1)^k\log^kq}{k! q}(z-1)^k,
\end{equation}
\begin{equation}\label{thm-p3}
\left(\frac{d}{dz}\right)^\ell zq^{-(z-1)}=(-1)^{\ell-1}\left(\ell\log^{\ell-1}q-z\log^\ell q\right)q^{-z-1},
\end{equation}
\begin{equation}\label{thm-p4}
q^{-z-k}=\frac1{q^{k+1}}e^{-(z-1)\log q}=\sum_{k=0}^\infty\frac{(-1)^k\log^kq}{k! q^{k+1}}(z-1)^k.
\end{equation}
Following Coffey \cite[p. 452, Lemma 1]{Co}, we  consider that
$$\left(\frac{d}{dz}\right)^j(z)_k\biggl|_{z=1}=(-1)^{k+j}j!s(k+1,j+1),$$
where $s(k,j)$ is the Stirling numbers of the first kind. Then applying Leibniz's rule and (\ref{thm-p4}), we get
\begin{equation}\label{thm-p5}
\begin{aligned}
\left(\frac{d}{dz}\right)^\ell(z)_kq^{-z-k}\biggl|_{z=1}&=\sum_{j=0}^\ell\binom\ell j
\left(\frac{d}{dz}\right)^j(z)_k\left(\frac{d}{dz}\right)^{\ell-j}q^{-z-k}\biggl|_{z=1} \\
&=\frac{(-1)^{k+\ell}}{q^{k+1}}\sum_{j=0}^\ell\binom\ell j j!s(k+1,j+1)\log^{\ell-j}q.
\end{aligned}
\end{equation}
From  (\ref{l-s-con}), we have $$(-1)^\ell\t\ga_\ell(q)=\left.\left(\frac{\partial}{\partial z}\right)^{\ell}\zeta_{E}(z,q)\right|_{z=1},$$ 
then taking $\ell$-th derivative with respect to $z$ on the both sides of  (\ref{thm-p1}), by (\ref{thm-p2}), (\ref{thm-p3}) and (\ref{thm-p5})
we get the following asymptotic expansion
$$\begin{aligned}
(-1)^\ell\t\ga_\ell(q)&\sim(-1)^\ell\frac{\log^\ell q}{2q}+(-1)^{\ell-1}\frac{\ell\log^{\ell-1}q-\log^\ell q}{4q^2} \\
&\quad-(-1)^\ell\sum_{k=2}^\infty(-1)^k\frac{E_k(0)}{2q^{k+1}k!}\sum_{j=0}^\ell\binom\ell jj! s(k+1,j+1)\log^{\ell-j}q
\end{aligned}$$
as $q\to\infty.$
Since $E_{2k}(0)=0$ for $k\geq1,$ we get our result.

It needs to mention that the term-wise differentiation of the asymptotic power series expansion (\ref{thm-p1}) is allowed by \cite[p. 608, Corollary 2]{Zorich}.

\subsection*{Proof of  Lemma \ref{lem1}}

For Re$(z+pk)>0$ and $q>0,$ by (\ref{m-tran}) we have
\begin{equation}\label{m-pk}
\Gamma(z+pk)\zeta_E(z+pk,q)=\int_0^\infty\frac{e^{(1-q)t}t^{z+pk-1}}{e^t+1}dt.
\end{equation}
This leads to the identity
\begin{equation}\label{m-pk1}
\sum_{k=1}^\infty\frac{2^{pk}\Gamma(z+pk)\zeta_E(z+pk,q)}{(pk+1)!p^{2k}}
=\int_0^\infty\frac{e^{(1-q)t}t^{z-1}}{e^t+1}\sum_{k=1}^\infty\frac{(2t)^{pk}}{(pk+1)!p^{2k}}dt,
\end{equation}
because the integral in (\ref{m-pk}) is absolutely convergent, and the interchange of integration and summation is allowed. The sum of the infinite series in the integral of the above equation  is  known in some cases, that is,
\begin{equation}\label{m-pk2}
\sum_{k=1}^\infty\frac{(2t)^{pk}}{(pk+1)!p^{2k}}
=\begin{cases}
\frac1{2t}e^{2t}-\frac1{2t}-1&\text{if }p=1, \\
\frac1t\sinh(t)-1&\text{if }p=2.
\end{cases}
\end{equation}
 (See \cite[(9a) and (9b)]{Co08}).

In the case $p=1,$ substituting  (\ref{m-pk2}) into (\ref{m-pk1}), then by (\ref{m-tran}) we have
\begin{equation}\label{m-pk3}
\begin{aligned}
\sum_{k=1}^\infty\frac{2^{k}\Gamma(z+k)\zeta_E(z+k,q)}{(k+1)!}
&=\int_0^\infty\frac{e^{(1-q)t}t^{z-1}}{e^t+1}\left(\frac1{2t}e^{2t}-\frac1{2t}-1\right)dt \\
&=\frac12\Gamma(z-1)(\zeta_E(z-1,q-2)-\zeta_E(z-1,q)) \\
&\quad-\Gamma(z)\zeta_E(z,q) \\
&=\frac12\Gamma(z-1)((q-2)^{1-z}-(q-1)^{1-z}) \\
&\quad-\Gamma(z)\zeta_E(z,q),
\end{aligned}
\end{equation}
the last equation follows from (\ref{J-1}) which is the difference equation of $\zeta_{E}(z,q)$. This concludes the proof of Part (1).

To see Part (2), note that
$$\frac1t\sinh(t)-1=\frac1{2t}(e^t-e^{-t})-1.$$
Thus by putting $p=2$ in (\ref{m-pk1}) and using (\ref{m-pk2}), we get Part (2) in a similar way with Part (1).

\subsection*{Proof of  Theorem \ref{thm4}}

To our purpose, we first  expand the first term on the right hand side of Lemma \ref{lem1}(1)  as power series in $z-1$: 
\begin{equation}\label{thm4-p1}
\begin{aligned}
&\frac{(q-2)^{1-z}-(q-1)^{1-z}}{2(z-1)} \\
&\quad=\frac12\sum_{k=1}^\infty(-1)^{k}\frac{\log^{k}(q-2)-\log^{k}(q-1)}{k!}(z-1)^{k-1} \\
&\quad=\frac12\sum_{k=0}^\infty(-1)^{k+1}\frac{\log^{k+1}(q-2)-\log^{k+1}(q-1)}{(k+1)!}(z-1)^{k}.
\end{aligned}
\end{equation}
For the second term, we compute its $\ell$-th derivative at $z=1$. By Leibniz's rule, we have
\begin{equation}\label{thm4-p2}
\begin{aligned}
&\left(\frac d{dz}\right)^\ell
\frac1{\Gamma(z)}\sum_{k=1}^\infty\frac{2^k\Gamma(z+k)}{(k+1)!}\zeta_E(z+k,q)\biggl|_{z=1} \\
&\quad=\sum_{k=1}^\infty\frac{2^k}{(k+1)!}\left(\frac d{dz}\right)^\ell(z)_k\zeta_E(z+k,q)\biggl|_{z=1} \\
&\quad=\sum_{k=1}^\infty\frac{2^k}{(k+1)!}\sum_{j=0}^\ell\binom\ell j
\left(\left(\frac d{dz}\right)^{\ell-j}(z)_k\right)\zeta_E^{(j)}(z+k,q)\biggl|_{z=1}.
\end{aligned}
\end{equation}
Lemma 1 of \cite{Co} provides us the derivative values
\begin{equation}\label{thm4-p3}
\left(\frac d{dz}\right)^\ell(z)_{pk}\biggl|_{z=1}=(-1)^{\ell+pk}\ell!s(pk+1,\ell+1),
\end{equation}
for $p=1,2,\ldots.$ From (\ref{thm4-p1}), (\ref{thm4-p2}), and (\ref{thm4-p3}) with $p=1,$ comparing with the Laurent series expansion of $\zeta_E(z,q)$ (see (\ref{l-s-con})), we obtain the identity in Part (1).

We can obtain the assertion of Part (2) from Lemma \ref{lem1}(2) and (\ref{thm4-p3}) with $p=2$ in a very similar way with Part (1).

\subsection*{Proof of  Proposition \ref{pro2}}

The proof is based on the following Fourier expansion of $\zeta_E(z,q):$
\begin{equation}\label{Fourier} \zeta_E(z,q)=\frac{2\Gamma(1-z)}{\pi^{1-z}}\sum_{n=0}^\infty\frac{\sin\left((2n+1)\pi q+\frac{\pi z}{2}\right)}{(2n+1)^{1-z}}.\end{equation}
This expression, valid for Re$(z) < 1$ and $q\neq0,-1,-2,\ldots,$ is obtained by Williams and Zhang \cite[p. 36, (1.7)]{WZ}.
The identity 
\begin{equation}\label{pro2-1}
\int_0^2\zeta_E^{(j)}(z,q)dq=0
\end{equation}
now follows   from (\ref{Fourier}) directly.
Taking $j$-th derivative with respect to $z$ on both sides of (\ref{l-s-con}), we have
\begin{equation}\label{pro2-2}
\zeta_E^{(j)}(z,q)=\sum_{k=j}^\infty\frac{(-1)^k\t\ga_k(q)}{(k-j)!}(z-1)^{k-j}.
\end{equation}
Substituting (\ref{pro2-2}) into (\ref{pro2-1}) and denote by  $s=1-z,$ we get
$$\sum_{k=j}^\infty\frac{s^k}{(k-j)!}\int_0^2\t\ga_k(q)dq=0,\quad\text{Re}(s)>0,$$
which is the conclusion of  Part (1). Setting $j=n$ and $z=0$ in (\ref{pro2-2}) we obtain Part (2). Note that
(see \cite[p. 40, Proposition 3]{WZ})
$$\zeta_E'(0,q)=\log\Gamma\left(\frac q2\right)-\log\Gamma\left(\frac{1+q}2\right)-\frac12\log 2.$$
Thus, Part (3) follows from Part (2) with $n=1.$

\subsection*{Proof of  Proposition \ref{pro3}}

From (\ref{J-2}) and (\ref{l-s-con}) it is easily seen that
$$\sum_{j=0}^\infty\frac{(-1)^j}{j!}\frac{d}{dq}\t\ga_j(q)(z-1)^j
=-\sum_{j=0}^\infty\frac{(-1)^j}{j!}\t\ga_j(q)z^{j+1}.$$
Then using the binomial theorem, we have
\begin{equation}\label{pro3-pf}
\begin{aligned}
\sum_{j=0}^\infty\frac{(-1)^j}{j!}\frac{d}{dq}\t\ga_j(q)(z-1)^j
&=-\sum_{j=0}^\infty\frac{(-1)^j}{j!}\t\ga_j(q)z^{j+1} \\
&=-\sum_{j=0}^\infty\frac{(-1)^j}{j!}\t\ga_j(q)\sum_{k=0}^{j+1}\binom{j+1}k(z-1)^k \\
&=-\sum_{j=0}^\infty\frac{(-1)^j}{j!}\t\ga_j(q) \\
&\quad-\sum_{j=1}^\infty\sum_{k=j-1}^\infty\frac{(-1)^k}{k!}\binom{k+1}j\t\ga_k(q)(z-1)^j.
\end{aligned}
\end{equation}
Equaling the coefficients on the both sides of the above equation,
we have
$$\frac{(-1)^j}{j!}\frac{d}{dq}\t\ga_j(q)=-\sum_{k=j-1}^\infty\frac{(-1)^k}{k!}\binom{k+1}j\t\ga_k(q)$$
for $j\geq 1$ and 
$$-\t\psi'(q)=\t\ga'_0(q)=-\sum_{k=0}^\infty\frac{(-1)^k}{k!}\t\ga_k(q)$$
for $j=0$ (see \ref{gamma0}),
which are the assertions of our proposition.

\section{Remarks and extensions}\label{Remarks}

Substituting (\ref{l-s-con}) into (\ref{w-form}), replacing $x$ by $-x$ and setting $q=1$, we get that
\begin{equation}\label{rem1}
\begin{aligned}
\sum_{k=0}^\infty\frac{(-1)^k}{k!}&\left(\t\ga_k(1-x)-\t\ga_k(1)\right)(z-1)^k \\
&=\sum_{j=1}^\infty\frac{\Gamma(z+j)}{\Gamma(z)j!}\left(\sum_{k=0}^\infty\frac{(-1)^k}{k!}\t\ga_k(z+j-1)^k\right)x^j
\end{aligned}
\end{equation}
for $|x|<1$.
Putting $z=1$ in (\ref{rem1}) we have
\begin{equation}\label{rem2}
\begin{aligned}
\t\ga_0(1-x)-\t\ga_0(1)
&=\sum_{j=1}^\infty x^j\sum_{k=0}^\infty\frac{(-1)^k}{k!}\t\ga_kj^k \\
&=\sum_{k=0}^\infty\frac{(-1)^k}{k!}\t\ga_k\sum_{j=1}^\infty x^jj^k
\end{aligned}
\end{equation}
for $|x|<1.$ In fact, the inner sum on $j$ in (\ref{rem2}) is a finite summation by notifying  the expansion
(see \cite[p. 2566, (3.10b)]{Co06} and \cite[p. 456, (4.10)]{Go}):
\begin{equation}\label{rem3}
\sum_{j=0}^\infty x^jj^k=\sum_{j=0}^kS(k,j)\frac{j!x^j}{(1-x)^{j+1}}, \quad|x|<1,
\end{equation}
where $S(k,j)$ are Stirling numbers of the second kind and $S(0,0)=1.$ Then by substituting (\ref{rem3}) into (\ref{rem2}) and notifying that $$\t\psi(q)=-\t\ga_0(q)$$ (see (\ref{psi-def})),
we get
\begin{equation}\label{rem4}
\t\psi(1)-\t\psi(1-x)=\sum_{k=0}^\infty\frac{(-1)^k}{k!}\t\ga_k
\left(\sum_{j=0}^kS(k,j)j!\frac{x^j}{(1-x)^{j+1}}-\delta_{k0}\right),
\end{equation}
where $|x|<1$ and $\delta_{k0}$ is the Kronecker symbol.
Recalling that (see, e.g., \cite[p. 20]{SC})
\begin{equation}\label{phi-re}
\psi\left(\frac12\right)=-2\log2-\ga,\quad\psi\left(\frac14\right)=-\frac\pi 2-3\log2-\ga,
\end{equation}
where $\gamma$ is Euler's constant.
So in (\ref{ps-ga}), taking $q=\frac12$ we obtain the  special value of $\t\psi$ at $\frac{1}{2}:$
\begin{equation}\label{phi-type}
\t\psi\left(\frac12\right)=-\psi\left(\frac12\right)+\psi\left(\frac14\right)+\log2=-\frac\pi2.
\end{equation}
By (\ref{gammak}), (\ref{psi-def}) and (\ref{remark14}), we get \begin{equation}\label{sec3-1}\t\psi(1)=-\t\ga_0(1)=-\t\ga_0=-\log 2.\end{equation}
Thus by (\ref{phi-type}) and (\ref{sec3-1}), we have $$\t\psi(1)-\t\psi(1/2)= \frac\pi2-\log2.$$
So evaluating   (\ref{rem4})  at $x=\frac12,$  we get
\begin{equation}\label{rem5}
2\sum_{k=1}^\infty\frac{(-1)^k}{k!}\t\ga_k\sum_{j=0}^kS(k,j)j!= \frac\pi2-\log2.
\end{equation}
This leads to the following series expansion of $\pi$.

\begin{proposition}\label{log2-pi}
$$\frac\pi2= \log2+2\sum_{k=1}^\infty\frac{(-1)^k}{k!}\t\ga_k\sum_{j=0}^kS(k,j)j!,$$
where $S(k,j)$ are Stirling numbers of the second kind and $S(0,0)=1.$ 
\end{proposition}

The following integral seems to be interesting. 

\begin{proposition}\label{hyp-int}
For ${\rm Re}(\delta)>-1,{\rm Re}(\beta)>-1,q>0,|x|<q$ and $|v|<1,$ we have
$$\begin{aligned}
\int_0^1t^\beta(1-t)^\delta&(1-tv)^{-\alpha}\left(\zeta_E(z,q-xt)-\zeta_E(z,q)\right)dt \\
&=\sum_{j=1}^\infty\frac{\Gamma(z+j)}{\Gamma(z)j!}B(\beta+j+1,\delta+1) \\
&\quad\times {}_2F_1(\alpha,\beta+j+1;\delta+\beta+j+2;v)\zeta_E(z+j,q)x^j,
\end{aligned}$$
where $B$ is the beta function and ${}_2F_1(a,b;c;v)$ is the Gauss hypergeometric function.
By additional imposing ${\rm Re}(\alpha+\beta-\delta-1)<0,$ convergence on the unit circe $|v|=1$
may be obtained.
\end{proposition}

\begin{remark}\label{hyp-int-re}
Coffey proved  a similar result for the Hurwitz zeta functions $\zeta(z,q)$ (see \cite[Proposition 4.1]{Co06}).
\end{remark}

\begin{proof}[Proof of Proposition \ref{hyp-int}]
By (\ref{w-form}), we have
\begin{equation}\label{w-form-re}
\zeta_E(z,q+x)-\zeta_E(z,q)=\sum_{j=1}^\infty\frac{(-1)^j}{j!}\frac{\Gamma(z+j)}{\Gamma(z)}\zeta_E(z+j,q)x^j.
\end{equation}
Letting $x\to -xt$ in (\ref{w-form-re}), multiplying both sides of (\ref{w-form-re}) by $t^\beta(1-t)^\delta(1-tv)^{-\alpha}$, then integrating on $t$ from 0 to 1, we obtain
\begin{equation}\label{w-form-re1}
\begin{aligned}
\int_0^1t^\beta(1-t)^\delta&(1-tv)^{-\alpha}\left(\zeta_E(z,q-xt)-\zeta_E(z,q)\right)dt \\
&=\sum_{j=1}^\infty\frac{\Gamma(z+j)}{\Gamma(z)j!}\zeta_E(z+j,q)x^j
\int_0^1t^{\beta+j}(1-t)^\delta(1-tv)^{-\alpha}dt.
\end{aligned}
\end{equation}
There is a well-known integral representation for ${}_2F_1(a,b;c;v),$ the reader may consult  to \cite[p. 46, (11)]{SC},
which is given by
\begin{equation}\label{w-form-re2}
\begin{aligned}
\int_0^1t^{b-1}(1-t)^{c-b-1}(1-vt)^{-a}dt
&= {}_2F_1(a,b;c;v)\frac{\Gamma(b)\Gamma(c-b)}{\Gamma(c)} \\
&= {}_2F_1(a,b;c;v)B(b,c-b),
\end{aligned}
\end{equation}
where ${\rm Re}(c)>{\rm Re}(b)>0,|\arg(1-v)|\leq\pi-\epsilon$ with $0<\epsilon<\pi$ and $B$ is the beta function.
If we put $a=\alpha,b=\beta+j+1,c=\delta+\beta+j+2$ in (\ref{w-form-re2}) and substitute it into the right hand side of (\ref{w-form-re1}), then we obtain our
proposition.
\end{proof}

In the case $\alpha=0,$ Proposition \ref{hyp-int}  reduces to a formula which is similar  with the main theorem of
Kanemitsu, Kumagai and Yoshimoto \cite{KK}.
As a special case, we have

\begin{corollary}\label{hyp-int-co}
For ${\rm Re}(\delta)>-1,{\rm Re}(\beta)>-1,q>0$ and $|x|<q,$ we have
$$\begin{aligned}
\int_0^1t^\beta(1-t)^\delta&\left(\zeta_E(z,q-xt)-\zeta_E(z,q)\right)dt \\
&=\sum_{j=1}^\infty\frac{\Gamma(z+j)}{\Gamma(z)j!}
\frac{\Gamma(\beta+j+1)\Gamma(\delta+1)}{\Gamma(\delta+\beta+j+2)}
\zeta_E(z+j,q)x^j.
\end{aligned}$$
\end{corollary}

\begin{remark}\label{hyp-int-re1}
(1) The special case  $\delta=0$ of Corollary \ref{hyp-int-co}  reads
\begin{equation}\label{re-hyp1}
\begin{aligned}
\int_0^1t^\beta\zeta_E(z,q-xt)dt =\sum_{j=1}^\infty\frac{\Gamma(z+j)}{\Gamma(z)j!}
\frac{\zeta_E(z+j,q)}{\beta+j+1}x^j+\frac{\zeta_E(z,q)}{\beta+1},
\end{aligned}
\end{equation}
where ${\rm Re}(\beta)>-1.$

(2) Another formula similar with Corollary \ref{hyp-int-co} is the following:
\begin{equation}\label{re-hyp2}
\begin{aligned}
\int_0^\infty t^\beta e^{-\alpha t}&\left(\zeta_E(z,q-xt)-\zeta_E(z,q)\right)dt \\
&=\sum_{j=1}^\infty\frac{\Gamma(z+j)}{\Gamma(z)j!}
\frac{\Gamma(\beta+j+1)}{\alpha^{\beta+j+1}}
\zeta_E(z+j,q)x^j,
\end{aligned}
\end{equation}
where ${\rm Re}(\alpha)>0$ and ${\rm Re}(\beta)>-1.$

In fact, by the integral expression
$$\int_0^\infty t^{\beta+j}e^{-\alpha t}dt=\frac1{\alpha^{\beta+j+1}}\int_0^\infty t^{\beta+j}e^{-t}dt,$$
we have
$$\int_0^\infty t^{\beta+j}e^{-\alpha t}dt=\frac1{\alpha^{\beta+j+1}}\Gamma(\beta+j+1).$$
So a procedure similar with  the proof of Proposition \ref{hyp-int} shows (\ref{re-hyp2}).
\end{remark}

\begin{proposition}\label{dis-1}
For odd positive integers $k$ and real $a,$ we have
\begin{equation}\label{sec3-2}
\begin{aligned}
\sum_{r=1}^{k}(-1)^{r-1}\t\ga_\ell\left(\frac rk-a\right)
=k\sum_{j=0}^\ell(-1)^j\binom\ell j\t\ga_{\ell-j}(1-ak)\log^j k,
\end{aligned}
\end{equation}
\begin{equation}\label{sec3-3}
\begin{aligned}
\sum_{r=1}^{k}(-1)^{r-1}\t\ga_\ell\left(\frac rk+a\right)
&=k\sum_{j=0}^{\ell-1}(-1)^j\binom\ell j\t\ga_{\ell-j}(1+ak)\log^j k \\
&\quad+k(-1)^\ell\t\psi(ak)\log^\ell k+\frac{(-1)^\ell}a\log^\ell k.
\end{aligned}
\end{equation}
\end{proposition}

\begin{proof}
Let $k$ be an odd positive integer. 
Recall the distribution formula of $\zeta_{E}(z,q)$ (see \cite[Lemma 3.1(2)]{KMS})
\begin{equation}\label{dis-pf}
\begin{aligned}
\sum_{r=0}^{k-1}(-1)^r\zeta_E\left(z,\frac rk+a\right)=k^z\zeta_E(z,ak),
\end{aligned}
\end{equation}
or, replacing $a$ by $\frac1k-a,$
\begin{equation}\label{dis-pf-1}
\begin{aligned}
\sum_{r=1}^{k-1}(-1)^{r-1}\zeta_E\left(z,\frac rk-a\right)=k^z\zeta_E(z,1-ak).
\end{aligned}
\end{equation}
Substituting the Laurent series expansion (\ref{l-s-con}) into (\ref{dis-pf-1}), and expanding  the exponential $k^z=ke^{(z-1)\log k}$ as  power series in $z-1$ on
the result equation,
we get
$$\begin{aligned}
\sum_{r=1}^k(-1)^{r-1}&\sum_{\ell=0}^\infty(-1)^\ell\t\ga_{\ell}\left(\frac rk-a\right)\frac{(z-1)^\ell}{\ell!} \\
&=\left(k\sum_{j=0}^\infty\log^jk\frac{(z-1)^j}{j!}\right)
\left(\sum_{\ell=0}^\infty(-1)^\ell\t\ga_{\ell}(1-ak)\frac{(z-1)^\ell}{\ell!} \right).
\end{aligned}$$
By comparing the coefficients on both sides of the above equation, we get (\ref{sec3-2}). Letting $a\to-a$ in (\ref{sec3-2}), separating the $j=\ell$ term of the sum on the right hand side,
and using the difference equation (\ref{ad-ga}) with $k=0$, we get (\ref{sec3-3}).
\end{proof}

\begin{remark}\label{dis-re}
Suppose that  $k$ is an odd positive integer. The special case $\ell=0$ of Proposition \ref{dis-1} reads
\begin{equation}\label{sec3-4} \sum_{r=1}^{k}(-1)^{r-1}\t\ga_0\left(a+\frac rk\right)=k\t\ga_0(1+ak).\end{equation}
Since by (\ref{psi-def}) $$\t\ga_0(q)=-\t\psi(q)$$ and by (\ref{ad-ga})
$$\t\ga_0(q+1)+\t\ga_0(q)=\frac1q,$$ (\ref{sec3-4}) is equivalent to
$$\sum_{r=1}^{k}(-1)^{r}\t\psi\left(a+\frac rk\right)=k\t\psi(ak)+\frac1a,$$
or
$$\sum_{r=0}^{k-1}(-1)^{r}\t\psi\left(a+\frac rk\right)=k\t\psi(ak).$$

Setting $k=3$ in the above equation and using  the known relation (e.g. \cite[Proposition 2.3, (2.14)]{BMM})
$$\t\psi(a)+\t\psi(1-a)=-\frac{\pi}{\sin\pi a}$$
at $a=\frac13,$
we  get the  special value of $\t\psi$ at $\frac{1}{3}:$
$$\t\psi\left(\frac13\right)=-\frac{\pi}{\sqrt3}-\log 2.$$
\end{remark}
The following series involving the modified Stieltjes constants $\t\gamma_{k}(q)$ and the alternating Hurwitz zeta function $\zeta_{E}(z,q)$ may be also interesting.
\begin{proposition}\label{dis-2}
Let $j\geq0$ be an integer and $|w|<1.$ We have
$$\begin{aligned}
\sum_{m=1}^\infty\sum_{k=j}^\infty\frac{(-1)^k}{(k-j)!}\t\gamma_k(q)m^{k-j}w^m
=\sum_{m=1}^\infty\zeta_E^{(j)}(m+1,q)w^m.
\end{aligned}$$
\end{proposition}

\begin{proof}
Replacing $z+1$ by $z$ in (\ref{l-s-con}), we obtain
$$\sum_{k=0}^\infty\frac{(-1)^k}{k!}\t\ga_k(q)z^k=\zeta_E(z+1,q).$$
Then differentiation of both sides of the above equation  $j$ times with respect to $z,$ 
and noticing that
$$\left(\frac{d}{dz}\right)^jz^k=\begin{cases}
j!\binom kj z^{k-j} &\text{if }k\geq j, \\
0 &\text{otherwise},
\end{cases}$$
we have
\begin{equation}\label{dis-2-pf-1}
\sum_{k=j}^{\infty}\frac{(-1)^k}{(k-j)!}\t\ga_k(q)z^{k-j}=\zeta_E^{(j)}(z+1,q).
\end{equation}
Setting  $z=m$ in (\ref{dis-2-pf-1}), multiplying  both sides of (\ref{dis-2-pf-1}) by $w^m,$ then summing over  $m$ from 1 to $\infty,$
we get our proposition.
\end{proof}

\begin{corollary}\label{dis-co}
Let $|w|<1.$ We have
$$\begin{aligned}
\sum_{m=1}^\infty\sum_{k=0}^\infty\frac{(-1)^k}{k!}\t\gamma_k(q)m^{k}w^m
&=\sum_{m=1}^\infty\zeta_E(m+1,q)w^m \\
&=\t\ga_0(q-w)-\t\ga_0(q).
\end{aligned}$$
\end{corollary}

\begin{proof}
This follows from (\ref{ell-0}) and Proposition \ref{dis-2}.
\end{proof}

We may note that Corollary \ref{dis-co}, a special case of Proposition \ref{dis-2}, recovers equation (\ref{rem2}).

\section*{Acknowledgement} The authors are enormously grateful to the anonymous referee for his/her very careful
reading of this paper, and for his/her many valuable and detailed suggestions. 
Su Hu is supported by Guangdong Basic and Applied Basic Research Foundation (No. 2020A1515010170).  Min-Soo Kim is supported by the National Research Foundation of Korea(NRF) grant funded by the Korea government(MSIT) (No. 2019R1F1A1062499).

\bibliography{central}

\end{document}